\documentclass[a4paper,leqno,12pt, reqno]{amsart}

\usepackage{amsmath}
\usepackage{amssymb}
\usepackage{graphicx}
\usepackage{tikz-cd}
\usepackage{hyperref}
\usepackage[all]{xy}
\usepackage{tikz}
\usepackage{xspace}
\usepackage{bm}
\usepackage{amsmath}
\usepackage{amstext}
\usepackage{amsfonts}
\usepackage[mathscr]{euscript}
\usepackage{amscd}
\usepackage{latexsym}
\usepackage{amssymb}
\usepackage{enumerate}
\usepackage{xcolor}
\usepackage{mathtools}

\theoremstyle{plain}
\swapnumbers

\newcommand{\cat}{{\sf{cat}}}

\newcommand{\secat}{{\sf{secat}}}

\newcommand{\TC}{{\sf{TC}}}

\newcommand{\hdim}{{\sf{hdim}}}
\newcommand{\Id}{{\sf{Id}}}

\newcommand{\rr}{\mathbb{R}}
\newcommand{\cc}{\mathbb{C}}

\renewcommand{\dim}{{\sf {dim}}}

\newtheorem{theorem}{Theorem}[section]
\newtheorem{prop}[theorem]{Proposition}
\newtheorem{lemma}[theorem]{Lemma}
\newtheorem{corollary}[theorem]{Corollary}

\theoremstyle{definition}
\newtheorem{definition}[theorem]{Definition}
\newtheorem{example}[theorem]{Example}

\newtheorem{subsec}[theorem]{}

\theoremstyle{remark}
\newtheorem{remark}[theorem]{Remark}

\begin{document}

	\title{Sequential Parametrized Motion Planning and its complexity}	
	\author{Michael Farber}
	
	\address{School of Mathematical Sciences\\
Queen Mary University of London\\ E1 4NS London\\UK.}
	
	\email{m.farber@qmul.ac.uk}
	
	\address{School of Mathematical Sciences\\
Queen Mary University of London\\ E1 4NS London\\UK.}
	
	\email{amit.paul@qmul.ac.uk/amitkrpaul23@gmail.com}
	
	%\date{\today}
	
	\thanks{Both authors were partially supported by an EPSRC research grant}
\author{Amit Kumar Paul}
	
	\begin{abstract} In this paper we develop a theory of {\it sequential parametrized} motion planning generalising the approach of {\it parametrized} motion planning, which was introduced recently in \cite{CohFW21}. 
A sequential parametrized motion planning algorithm produced a motion of the system which is required to visit a prescribed sequence of states, in a certain order, at specified moments of time. 
The sequential parametrized algorithms are universal as the external conditions are not fixed in advance but rather constitute part of the input of the algorithm. In this article we give a detailed analysis of the sequential parametrized topological complexity of the Fadell - Neuwirth fibration. In the language of robotics, sections of the Fadell - Neuwirth fibration are  algorithms for moving multiple robots avoiding collisions with other robots and with obstacles in the Euclidean space. 
In the last section of the paper we introduce the new notion of {\it $\TC$-generating function of a fibration}, examine examples and 
raise some exciting general questions about its analytic properties. 
\end{abstract}
	\maketitle
\begin{section}{Introduction}

Autonomously functioning systems in robotics are controlled by motion planning algorithms. 
Such an algorithm takes as input the initial and the final states of the system and produces a motion of the system from the initial to final state, as output. The theory of algorithms for robot motion planning is a very active field of contemporary 
robotics and we refer the reader to the monographs \cite{Lat}, \cite{LaV} for further references.

A topological approach to the robot motion planning problem was developed in \cite{Far03}, \cite{Far04}; the topological techniques explained relationships between instabilities occurring in robot motion planning algorithms and topological features of robots' configuration spaces.

A new {\it parametrized} approach to the theory of motion planning algorithms was suggested recently in \cite{CohFW21}. The parametrized algorithms are more universal and flexible, they can function in a variety of situations involving external conditions which are viewed as parameters and are part of the input of the algorithm. A typical situation of this kind arises when we are dealing with collision-free motion of many objects (robots) moving in the 3-space avoiding a set of obstacles, and the positions of the obstacles are a priori unknown. This specific problem was analysed in full detail in \cite{CohFW21}, \cite{CohFW}. 

In this paper we develop a more general theory of {\it sequential parametrized} motion planning algorithms. In this approach the algorithm produces a motion of the system which is required to visit a prescribed sequence of states in a certain order.
The sequential parametrized algorithms are also universal as the external conditions are not a priori fixed but constitute a part of the input of the algorithm. 

In the first part of this article we develop the theory of sequential parametrized motion planning algorithms while the second part consists of a detailed analysis of the sequential parametrized topological complexity of the Fadell - Neuwirth fibration.
 In the language of robotics, the sections of the Fadell - Neuwirth bundle are exactly the algorithms for moving multiple robots avoiding collisions with each other and with multiple obstacles in the Euclidean space. 

Our results depend on the explicit computations of the cohomology algebras of certain configuration spaces. We describe these computations in full detail, they employ the classical Leray - Hirsch theorem from algebraic topology of fibre bundles. 

In the last section of the paper we introduce the new notion of a {\it $\TC$-generating function of a fibration}, discuss a few examples, and raise interesting questions about analytic properties of this function. 

In a forthcoming publication (which is now in preparation) we shall describe the explicit sequential parametrized 
motion planning algorithm for collision free motion of multiple robots in the presence of multiple obstacles in $\rr^d$, generalising  the ones presented in \cite{FarW}. 
These algorithms are optimal as they have minimal possible topological complexity. 

\end{section}

\begin{section}{Preliminaries}
In this section we recall the notions of sectional category and topological complexity; we refer to 
\cite{BasGRT14, CohFW21, CohFW, Far03, Gar19, Rud10, Sva66} for more information.
\subsection*{Sectional category}
Let $p: E \to B$ be a Hurewicz fibration. The \emph{sectional category} of $p$, denoted $\secat[p: E \to B]$ or $\secat(p)$, is defined as the least non-negative integer $k$ such that there exists an open cover 
$\{U_0, U_1, \dots, U_k\}$ of $B$ 
with the property that each open set $U_i$ admits a continuous section $s_i: U_i \to E$ of $p$. 
We set $\secat(p)=\infty$ if no finite $k$ with this property exists.

The \emph{generalized sectional category} of a Hurewicz fibration $p: E \to B$, denoted $\secat_g[p: E \to B]$ or $\secat_g(p)$, is defined as the least non-negative integer $k$ such that $B$ admits a partition $$B=F_0 \sqcup F_1 \sqcup ... \sqcup F_k, \quad F_i\cap F_j = \emptyset \text{ \ for \ } i\neq j$$ 
with each set $F_i$ admitting a continuous section $s_i: F_i \to E$ of $p$.  
We set $\secat_g(p)=\infty$ if no such finite $k$ exists.

It is obvious that $\secat(p) \geq \secat_g(p)$ in general. However, as was established in \cite{Gar19},
in many interesting situations there is an equality:

\begin{theorem}
\label{lemma secat betweeen ANR spaces}
Let $p : E \to B$ be a Hurewicz fibration with $E$ and $B$ metrizable absolute neighborhood retracts (ANRs). 
Then $\secat(p) = \secat_g(p).$
\end{theorem}

In the sequel the term \lq\lq fibration\rq\rq \ will always mean \lq\lq Hurewicz fibration\rq\rq, unless otherwise stated explicitly. 

The following Lemma will be used later in the proofs. 

\begin{lemma}\label{lm:secat} (A) If for two fibrations $p: E\to B$ and $p':E'\to B$ over the same base $B$ there exists a continuous map $f$ shown on the following digram 
\begin{center}
$\xymatrix{
 E  \ar[dr]_{ p} \ar[rr]^{f}  &&  E' \ar[dl]^{p'} \\
& B
}$
\end{center}
then $\secat(p)\ge \secat(p')$. 

(B) If a fibration $p: E\to B$ can be obtained as a pull-back from another fibration $p': E'\to B'$ then 
$\secat(p)\le \secat(p')$. 

(C) Suppose that for two fibrations $p: E\to B$ and $p': E'\to B'$ there exist continuous maps $f, g, F, G$ shown on the commutative diagram
\begin{center}
$\xymatrix{
E\ar[r]^F\ar[d]_p & E'\ar[d]^{p'} \ar[r]^G & E\ar[d]^p\\
B \ar[r]_f & B'\ar[r]_{g}& B,
}$
\end{center}
such that $g\circ f: B\to B$ is homotopic to the identity map $\Id_B: B\to B$. Then 
$\secat(p)\le \secat(p')$. 
\end{lemma}
\begin{proof} Statements (A) and (B) are well-known and follow directly from the definition of sectional category. Below we give the proof of (C) which uses (A) and (B).

Consider the fibration $q: \bar E \to B$ induced by $f: B\to B'$ from $p': E'\to B'$. Here $\bar E=\{(b, e')\in B\times E'; f(b)=p'(e')\}$ and $q(b, e')=b$.  Then 
\begin{eqnarray}\label{sec11}
\secat(q)\le \secat(p')
\end{eqnarray} 
by statement (B). 

Consider the map $\bar G: \bar E\to E$ given by $\bar G(b, e')= G(e')$ for $(b, e')\in \bar E$. Then 
$$(p\circ \bar G)\ (b, e') = p(G(e')) = g(p'(e')) = g(f(b)) = ((g\circ f)\circ q)\ (b, e')$$
and thus $p\circ \bar G = (g\circ f)\circ q $ and using the assumption $g\circ f\simeq \Id_B$ we obtain $p\circ \bar G \simeq q$. 
%We know that the map $g\circ f: B\to B$ is homotopic to the identity map $\Id: B\to B$. 
Let $h_t: B\to B$ be a homotopy with $h_0=g\circ f$ and $h_1=\Id_B$, $t\in I$. 
Using the homotopy lifting property, we obtain a homotopy
$\bar G_t: \bar E\to E$, such that $\bar G_0=\bar G$ and $p\circ \bar G_t = h_t\circ q$. The map $\bar G_1: \bar E \to E$ satisfies $p\circ \bar G_1= q$; in other words, $\bar G_1$ appears in the commutative diagram
\begin{center}
$\xymatrix{
 \bar E  \ar[dr]_{ q} \ar[rr]^{\bar G_1}  &&  E \ar[dl]^{p} \\
& B
}$
\end{center}
Applying to this diagram statement (A) we obtain the inequality $\secat(p) \le \secat(q)$ which together with inequality (\ref{sec11}) implies $\secat(p) \le \secat(p')$, as claimed. 
\end{proof}

\subsection*{Topological complexity}
Let $X$ be a path-connected topological space. Consider the path space $X^I$ (i.e. the space of all continuous 
maps $I=[0,1] \to X$ equipped with compact-open topology) and the fibration 
$$\pi : X^I \to X \times X, \quad \alpha \mapsto (\alpha(0), \alpha(1)).$$ 
The {\it topological complexity} $\TC(X)$ of $X$ is defined as $\TC(X):=\secat(\pi)$, cf. \cite{Far03}. For information on recent developments related to the notion of $\TC(X)$ we refer the reader to \cite{GraLV}, \cite{Dra}. 

% where $\secat(\pi)$ is the sectional category of $\pi$. Recall that the sectional category of a fibration $p: E \to B$ is the least non-negative integer $k$ such that there is an open cover $\{U_i\}_{i=0}^k$ of $B$ and on each open set $U_i$ there exists a continuous section $s_i: U_i \to E$ of $p$. 

For any $r\ge 2$, fix $r$ points $0\le t_1<t_2<\dots <t_r\le 1$ (which we shall call the {\it \lq\lq time schedule\rq\rq}) and consider the evaluation map 
\begin{eqnarray}\label{pir}
\pi_r : X^I \to X^r, \quad \alpha \mapsto \left(\alpha(t_1), \alpha(t_2), \dots,  \alpha(t_r)\right), \quad \alpha \in X^I.
\end{eqnarray}
Typically, one takes $t_i=(i-1)(r-1)^{-1}$. 
The {\it $r$-th sequential topological complexity} is defined as $\TC_r(X):=\secat(\pi_r)$; this invariant was originally introduced by Rudyak \cite{Rud10}. 
It is known that $\TC_r(X)$ is a homotopy invariant, it vanishes if and only if the space $X$ is contractible. Moreover, 
$\TC_{r+1}(X)\ge \TC_r(X)$. Besides, $\TC(X)=\TC_2(X)$.

\subsection*{Parametrized topological complexity} 
For a Hurewicz fibration $p : E \to B$ denote by $E^I_B\subset E^I$ the space of all paths $\alpha: I\to E$ such that $p\circ \alpha: I\to B$ is a constant path. Let $E^2_B\subset E\times E$ denote the space of pairs $(e_1, e_2)\in E^2$ satisfying $p(e_1)=p(e_2)$. 
Consider the fibration 
$$\Pi: E^I_B \to E^2_B = E\times_B E, \quad \alpha \mapsto (\alpha(0), \alpha(1)).$$ 
The fibre of $\Pi$ is the loop space $\Omega X$ where $X$ is the fibre of the original fibration $p:E\to B$. 
The following notion was introduced in a recent paper \cite{CohFW21}:
\begin{definition} 
The {\it parametrized topological complexity} $\TC[p : E \to B]$ of the fibration $p : E \to B$ is defined as
$$\TC[p : E \to B]=\secat[\Pi: E^I_B \to E^2_B].$$
%In other words, $\TC[p : E \to B]$ is 
%the minimal $k$ such that there is an open cover $\{U_i\}_{i=0}^k$ of $E^2_B$ and on each open set there is a section $s_i : U_i \to E^I_B$ of $\Pi$. In other words 
\end{definition}
Parametrized motion planning algorithms are universal and flexible, they are capable to function 
under a variety of external conditions which are parametrized by the points of the base $B$. 
We refer to \cite{CohFW21} for more detail and examples. 

If $B'\subset B$ and $E'=p^{-1}(B')$ then obviously $\TC[p : E \to B] \geq \TC[p' : E' \to B']$ where $p'=p|_{E'}$. In particular, restricting to a single fibre we obtain $$\TC[p : E \to B] \geq \TC(X).$$
\end{section}	
	
	\section{The concept of sequential parametrized topological complexity}
	
In this section we define a new notion of sequential parametrized topological complexity and establish its basic properties.

Let $p : E \to B$ be a Hurewicz fibration with fibre $X$. Fix an integer $r\ge 2$ and denote 
$$E^r_B= \{(e_1, \cdots, e_r)\in E^r; \, p(e_1)=\cdots = p(e_r)\}.$$ 
Let $E^I_B\subset E^I$ be as above 
the space of all paths $\alpha: I\to E$ such that $p\circ \alpha: I\to B$ is constant. Fix $r$ points 
$$0\le t_1<t_2<\dots <t_r\le 1$$ in $I$ (for example, one may take $t_i=(i-1)(r-1)^{-1}$ for $i=1, 2, \dots, r$), which will be called the {\it time schedule}.
Consider the evaluation map 
\begin{eqnarray}\label{Pir}
\Pi_r : E^I_B \to E^r_B, \quad \Pi_r(\alpha) = (\alpha(t_1), \alpha(t_2), \dots,  \alpha(t_r)).\end{eqnarray} 
$\Pi_r$ is a Hurewicz fibration, see \cite[Appendix]{CohFW}, the  fibre of $\Pi_r$ is $(\Omega X)^{r-1}$. 
A section $s: E^r_B \to E^I$ of the fibration $\Pi_r$ can be interpreted as a parametrized motion planning algorithm, i.e. as 
a function which assigns to every sequence of points $(e_1, e_2, \dots, e_r)\in E^r_B$ a continuous path $\alpha: I\to E$ (motion of the system) satisfying $\alpha(t_i)=e_i$ for every $i=1, 2, \dots, r$ and such that the path 
$p\circ \alpha: I \to B$ is constant. The latter condition means that the system moves under constant external conditions 
(such as positions of the obstacles). 

Typically $\Pi_r$ does not admit continuous sections; then the 
motion planning algorithms are necessarily discontinuous. 

The following definition gives a measure of complexity of sequential parametrized motion planning algorithms. This concept is the main character of this paper. 

\begin{definition}\label{def:main}
The {\it $r$-th sequential parametrized topological complexity} of the fibration $p : E \to B$, denoted $\TC_r[p : E \to B]$, is defined as the sectional category of the fibration $\Pi_r$, i.e. 
\begin{eqnarray}
\TC_r[p : E \to B]:=\secat(\Pi_r).
\end{eqnarray}
\end{definition}	

In more detail, $\TC_r[p : E \to B]$ is the minimal integer $k$ such that there is a open cover 
$\{U_0, U_1, \dots, U_k\}$ of $E^r_B$ with the property that each open set $U_i$ 
admits a continuous section $s_i : U_i \to E^I_B$ of $\Pi_r$.

Let  $B'\subset B$ be a subset and let $E'=p^{-1}(B')$ be its preimage, then obviously $$\TC_r[p: E \to B]\geq \TC_r[p': E' \to B']$$ 
where $p'=p|_{E'}$. 
In particular, taking $B'$ to be a single point, we obtain $$\TC_r[p: E \to B] \geq \TC_r(X),$$ where $X$ is the fibre of $p$.

\begin{example}\label{example para tc trivial fibration}
	Let $p: E \to B$ be a trivial fibration with fibre $X$, i.e. $E=B\times X$. 
	In this case we have $E^r_B=B\times X^r$, $E^I_B= B\times X^I$ and the map $\Pi_r: E_B^I\to E^r_B$ becomes
	$$\Pi_r: B\times X^I\to B\times X^r, \quad \Pi_r= \Id_B\times \pi_r,$$
	where $\Id_B: B\to B$ is the identity map and $\pi_r$ is the fibration (\ref{pir}). Thus we obtain in this example
	$$\TC_{r}[p: E \to B]= \TC_r(X),$$
	i.e. for the trivial fibration the sequential parametrized topological complexity equals the sequential topological complexity of the fibre. 
	
%	We have $$E^r_B = \{(x_1, b), (x_2, b) \cdots, (x_r, b) \ | \ x_1, x_2, \cdots, x_r \in X \text{ and } b\in B\} = X^r \times B,$$ $$E^I_B = \{\gamma = (\gamma_X, \gamma_B) : [0, 1]\to X \times B \ | \ \gamma_B \text{ is constant }\} = X^I \times B,$$ and the map $\Pi_r:E_B^I \to E_B^r$ is given by $e_r  \times \Id :X^I \times B \to X^r \times B$ where $e_r  :X^I \to X^r$ is given by $e_r(\gamma_X)= (\gamma_X(0), \gamma_X(\frac{1}{r-1}), \gamma_X(\frac{2}{r-1}), \cdots, \gamma_X(1))$. We obtain an isomorphism of fibrations.	
%$$\xymatrix{
%E_B^I \ar[dd]_{\Pi_r} \ar@{=}[rr] &&  X^I \times B \ar[dd]^{e_r \times \Id} \\\\
%E_B^r \ar@{=}[rr] && X^r \times B
%}$$
%Thus $\secat(\Pi_r)=\secat(e_r \times \Id)=\secat(e_r)=\TC_r(X)$ and we have $$\TC_{r}[p: E \to B]= \TC_r(X).$$
\end{example}
%\begin{lemma}\label{lemma cat} For any space $X$, $\cat(X^r)=secat(e_0)$, where $e_0: P_0X \to X^r, \gamma \to (\gamma(\frac{1}{n}), \gamma(\frac{2}{n}), \cdots, \gamma(1))$ and $P_0X$ is the path space starts with some fix poit $x_0$. \end{lemma} \begin{proof} Let $U$ be a open subset of $X^r$ ad there is a section $s : U\to P_0X$ of $e_0$. Define a homotopy $H : U \times I \to X^r$ by $$H(x_1, x_2, \cdots, x_r, t)= (s(\frac{t}{r}), s(\frac{2t}{r}), s(\frac{3t}{r}), \cdots,  s(t)).$$ Then $H$ is a homotopy between constant map and inclusion map on $U$. \\ Conversely, assume that there is homotopy $H$ from constant map at $(x_0, \cdots, x_0)$ to inclusion map on $U$. So for any $(x_1, x_2, \cdots, x_r)\in U$ we get paths $h_j(t)$ from $x_0$ to $x_j$ for $1\leq j \leq r.$ We define  $s : U\to P_0X$ by $$s(x_1, \cdots, x_r)=h_1(t/2)*h_1^{-1}(t)*h_2(t)*h_2^{-1}(t)* \cdots *h_{r-1}(t)*h_{r-1}^{-1}(t)*h_r(t).$$ Then $s$ is a section for $e_n$ on $U$.\end{proof}

\begin{prop}\label{princ}
	Let $p: E \to B$ be a principal bundle with a connected topological group $G$ as fibre. Then $$\TC_{r}[p: E \to B] = \cat(G^{r-1})=\TC_r(G).$$
\end{prop}

\begin{proof} Let $0\le t_1<t_2<\dots < t_r\le 1$ be the fixed time schedule. 
Denote by $P_0G\subset G^I$ the space of paths $\alpha$ satisfying $\alpha(t_1)=e$ where $e\in G$ denotes the unit element. Consider the evaluation map $\pi_r': P_0G\to G^{r-1}$ where $\pi'_r(\alpha)=(\alpha(t_2), \alpha(t_3), \dots, \alpha(t_r))$.  We obtain the commutative diagram 
\begin{center}
$\xymatrix{
P_0G \times E  \ar[d]_{ \pi'_r \times \Id } \ar[r]^{F}  &  E^I_B \ar[d]^{\Pi_r} \\
G^{r-1} \times E \ar[r]_{F'} & E^r_B
}$
\end{center}
where
 $F: P_0G \times E  \to  E^I_B$ and $F': G^{r-1}\times E\to E^r_B$ are homeomorphisms given by
 $$F(\alpha, x)(t)= \alpha(t)x, \quad F'(g_2, g_3, \dots, g_r, x)= (x, g_2 x, g_3 x, \dots, g_r x),$$
 where $\alpha \in P_0G$, \, $x\in E$, \, $t\in I$ and $g_i\in G$. 
Thus we have
$$
\TC_r[p: E\to B]= \secat(\Pi_r)=\secat(\pi'_r \times \Id )=\secat(\pi'_r).
$$
Clearly, $\secat(\pi'_r)=\cat(G^{r-1})$ since $P_0G$ is contractible. 
And finally $\cat(G^{r-1})=\TC_r(G)$, see  
 \cite[Theorem 3.5]{BasGRT14}.
\end{proof}	

\begin{example} As a specific example consider the Hopf fibration $p: S^3 \to S^2$ with fibre $S^1$. Applying the result of the previous Proposition we obtain 
$$\TC_r[p: S^3 \to S^2]=\TC_r(S^1)=r-1$$
for any $r\ge 2$. 	
\end{example}	
\subsection*{Alternative descriptions of sequential parametrized topological complexity}
Let $K$ be a path-connected finite CW-complex and let $k_1, k_2, \cdots, k_r\in K$ be a collection of $r$ pairwise 
distinct points of $K$, where $r\ge 2$. For a Hurewicz fibration $p: E \to B$, consider the space $E^{K}_B$ of all continuous maps 
$\alpha: K \to E$ such that the composition $p\circ \alpha: K\to B$ is a constant map. 
We equip $E^K_B$ with the compact-open topology induced from the function space $E^K$. 
Consider the evaluation map
$$\Pi_r^K : E^{K}_B \to E^r_B, \quad \Pi^K_r(\alpha) = (\alpha(k_1), \alpha(k_2), \cdots, \alpha(k_r)) \quad \mbox{for}
\quad\alpha\in E^K_B.$$ 
It is known that $\Pi^K_r$ is a Hurewicz fibration, see Appendix to \cite{CohFW}.

\begin{lemma}\label{lemma para tc by secat} For any path-connected finite CW-complex $K$
and a set of pairwise distinct points $k_1, \dots, k_r\in K$ one has $$\secat(\Pi^K_r) = \TC_r[p:E\to B].$$
 \end{lemma}		
\begin{proof} Let $0\le t_1<t_2<\dots<t_r\le 1$ be a given time schedule used in the definition of the map $\Pi_r=\Pi_r^I$ given by (\ref{Pir}). 
Since $K$ is path-connected we may find a continuous map $\gamma: I\to K$ with $\gamma(t_i) =k_i$ for all $i=1, 2, \dots, r$. We obtain a continuous map $F_\gamma: E^K_B \to E^I_B$ acting by the formula $F_\gamma(\alpha) = \alpha \circ \gamma$. It is easy to see that the following diagram commutes
$$
 \xymatrix{
E_B^{K} \ar[rr]^{F}  \ar[dr]_{\Pi_{r}^K}& &E_B^{I} \ar[dl]^{\Pi_{r}^I} \\ & E_B^r
}$$
Using statement (A) of Lemma \ref{lm:secat} we obtain
%Any partial section $s: U\to E^K_r$ of $\Pi_r^K$ defines a partial section $F\circ s$ of $\Pi_r^I$ implying
$$\TC_r[p:E\to B]= \secat(\Pi^I_r)\le \secat(\Pi_r^K).$$
To obtain the inverse inequality note that any locally finite CW-complex is metrisable. Applying Tietze extension 
theorem we can find continuous functions $\psi_1, \dots, \psi_r: K\to [0,1]$ such that $\psi_i(k_j)=\delta_{ij}$, 
i.e. $\psi_i(k_j)$ equals 1 for $j=i$ and it equals $0$ for $j\not=i$. The function $f=\min\{1, \sum_{i=1}^r t_i\cdot \psi_i\}: K\to [0,1]$
has the property that $f(k_i)=t_i$ for every $i=1, 2, \dots, r$. We obtain a continuous map $F': E^I_B \to E^K_B$, where $F'(\beta) = \beta\circ f$, \, $\beta\in E^I_B$, which appears in the commutative diagram
$$
 \xymatrix{
E_B^{I} \ar[rr]^{F'}  \ar[dr]_{\Pi_{r}^I}& &E_B^{K} \ar[dl]^{\Pi_{r}^K} \\ & E_B^r
}$$
By Lemma \ref{lm:secat} this implies the opposite inequality 
$\secat(\Pi_r^K) \le \secat(\Pi^I_r)$
and completes the proof. 
\end{proof}

 The following proposition is an analogue of \cite[Proposition 4.7]{CohFW21}.
\begin{prop}
Let $E$ and $B$ be metrisable separable ANRs and let $p: E \to B$ be a locally trivial fibration. Then the sequential parametrized topological complexity $\TC_r[p: E \to B]$ equals the smallest integer $n$ such that $E_B^r$ admits a partition $$E_B^r=F_0 \sqcup F_1 \sqcup ... \sqcup F_n, \quad F_i\cap F_j= \emptyset \text{ \ for } i\neq j,$$
with the property that on each set $F_i$ there exists a continuous section $s_i : F_i \to E_B^I$ of $\Pi_r$. In other words, $$\TC_r[p : E \to B] = \secat_g[\Pi_r: E_B^I \to E^r_B].$$
\end{prop}

\begin{proof}

From the results of \cite[Chapter IV]{Bor67} it follows that the fibre $X$ of $p: E \to B$ is ANR and hence $X^r$ is also ANR. Now, $E^r_B$ is the total space of the locally trivial fibration $E_B^r \to B$  with fibre $X^r$. Thus, applying \cite[Chapter IV, Theorem 10.5]{Bor67}, we obtain that the space $E^r_B$ is ANR. Using \cite[Proposition 4.7]{CohFW21} we see that $E_B^I$ is ANR. 
Finally, using Theorem \ref{lemma secat betweeen ANR spaces}, we conclude that $\TC_r[p : E \to B] = \secat_g[\Pi_r: E_B^I \to E^r_B].$
\end{proof}

\section{Fibrewise homotopy invariance}

\begin{prop}\label{prop homotopy invariant sptc}
	Let $p : E \to B$ and $p': E' \to B$ be two fibrations and let $f: E\to E'$ and $g: E'\to E$ be two continuous maps
	such the following diagram commutes
	$$
 \xymatrix{
E \ar@<-1.4pt>[rr]^{g}  \ar[dr]_{p}& &E'  \ar@<-1.4pt>[ll]^{f}  \ar[dl]^{p'} \\ & B
}$$	
i.e. $p=p'\circ f$ and $p'=p\circ g$. 
If the map $g \circ f : E \to E$ is fibrewise homotopic to the identity map $\Id_{E}: E \to E$ then 
$$\TC_r[p: E \to B] \leq \TC_r[p': E' \to B].$$
\end{prop}
\begin{proof} Denote by $f^r: E^r_B\to E'^r_B$ the map given by $f^r(e_1, \dots, e_r) =(f(e_1), \dots, f(e_r))$ and by 
$f^I: E^I_B \to E'^I_B$ the map given by $f^I(\gamma)(t)= f(\gamma(t))$ for $\gamma\in E^I_B$ and $t\in I$. One defines similarly the maps $g^r: E'^r_B\to E^r_B$ and $g^I: E'^I_B \to E^I_B$. This gives the commutative diagram 
\begin{center}
$\xymatrix{
E^I_B\ar[r]^{f^I}\ar[d]_{\Pi_r} & E'^I_B\ar[d]^{\Pi'_r} \ar[r]^{g^I} & E^I_B\ar[d]^{\Pi_r}\\
E^r_B \ar[r]_{f^r} & E'^r_B\ar[r]_{g^r}& E^r_B,
}$
\end{center}
in which $g^r\circ f^r \simeq \Id_{E^r_B}$. Applying statement (C) of Lemma \ref{lm:secat} we obtain 
\begin{eqnarray*}
\TC_r[p:E\to B]&=&\secat[\Pi_r: E^I_B \to E^r_B]\\ &\le& \secat[\Pi'_r: E'^I_B \to E'^r_B] \\ &=& \TC_r[p':E'\to B].
\end{eqnarray*}
\end{proof}
Proposition \ref{prop homotopy invariant sptc} obviously implies the following property of 
$\TC_r[p:E\to B]$:

\begin{corollary}\label{fwhom}
If fibrations $p : E \to B$ and $p' : E' \to B$ are fibrewise homotopy equivalent then $$\TC_r[p: E \to B] = \TC_r[p': E' \to B].$$
\end{corollary}

\section{Further properties of $\TC_r[p: E\to B]$}
Next we consider products of fibrations: 

\begin{prop}\label{prop product inequality}
Let $p_1: E_1 \to B_1$ and $p_2: E_2 \to B_2$ be two fibrations where the spaces $E_1, E_2, B_1, B_2$ are metrisable.
Then for any $r\geq 2$ we have 
$$\TC_r[p_1\times p_2 : E_1 \times E_2 \to B_1\times B_2]\leq \TC_r[p_1 : E_1 \to B_1] + \TC_r[p_2 : E_2 \to B_2].$$

\end{prop}

\begin{proof}
The proof is essentially identical to the proof of \cite[Proposition 6.1]{CohFW21} where it is done for the case $r=2.$
\end{proof}

\begin{prop}
Let $p_1: E_1 \to B$ and $p_2: E_2 \to B$ be two fibrations where the spaces $E_1, E_2, B$ are metrisable. 
Consider the fibration 
$p: E\to B$ where $E=E_1\times_B E_2=\{(e_1, e_2)\in E_1 \times E_2 \ | \ p_1(e_1) = p_2(e_2)\}$ 
and $p(e_1, e_2)=p(e_1)=p(e_2)$. 
Then $$\TC_r[p: E \to B]\, \leq \, \TC_r[p_1 : E_1 \to B] + \TC_r[p_2 : E_2 \to B].$$
\end{prop}

\begin{proof} Viewing $B$ as the diagonal of $B\times B$ gives 
$$\TC_r[p: E \to B]\leq \TC_r[p_1\times p_2 : E_1 \times E_2 \to B\times B].$$ 
Combining this inequality with the result of Proposition \ref{prop product inequality} completes the proof.
\end{proof}

\begin{lemma}
For any fibration $p: E \to B$ one has $$\TC_{r+1}[p: E \to B]\geq \TC_{r}[p: E \to B].$$
\end{lemma}

\begin{proof} We shall apply Lemma \ref{lemma para tc by secat} and consider the interval $K=[0,2]$ and the 
time schedule
$0\le t_1< t_2< \dots< t_r\le 1$ and the additional point $t_{r+1}=2.$ We have the following diagram
\begin{center}
$\xymatrix{
E^I_B\ar[r]^{F}\ar[d]_{\Pi_r} & E^K_B\ar[d]^{\Pi^K_{r+1}} \ar[r]^{G} & E^I_B\ar[d]^{\Pi_r}\\
E^r_B \ar[r]_{f} & E^{r+1}_B\ar[r]_{g}& E^r_B,
}$
\end{center}
where $f$ acts by the formula $f(e_1, e_2, \dots, e_r) = (e_1, e_2, \dots, e_r, e_r)$ and $F$ sends a path 
$\gamma: I\to E$, $\gamma\in E^I_B$, to the path $\bar \gamma: K=[0,2]\to E$ where $\bar\gamma|_{[0,1]}=\gamma$ and
$\bar \gamma(t) = \gamma(1)$ for any $t\in [1, 2]$. The vertical maps are evaluations at the points $t_1, \dots, t_r$ and 
at the points $t_1, \dots, t_r, t_{r+1}$, for 
$\Pi_r$ and $\Pi^K_{r+1}$ correspondingly. The map $G$ is the restriction, it maps $\alpha: K\to E$ to $\alpha_I: I\to E$. Similarly, the map $g: E^{r+1}_B \to E^r_B$ is given by $(e_1, \dots, e_r, e_{r+1}) \mapsto (e_1, \dots, e_r)$. The diagram commutes and besides the composition $g\circ f: E^r_B\to E^r_B$ is the identity map. Applying statement (C) of Lemma \ref{lm:secat} we obtain
\begin{eqnarray*}
\TC_r[p:E\to B]&=&\secat[\Pi_r: E^I_B\to E^r_B] \\ &\le& \secat[\Pi^K_{r+1}: E^K_B\to E^{r+1}_B]\\ &=& \TC_{r+1}[p:E\to B].
\end{eqnarray*}\end{proof}

%Let $\TC_{r+1}[p: E \to B] = k$ and $\{U_i:~ i=1,\cdots, k\}$ be an open cover of $E_B^{r+1}$ such that for each $i$ there is a section $s_i : U_i \to E_B^{I_{r+1}}$ of $\Pi'_{r+1}: E_B^{I_{r+1}} \to E_B^{r+1}$. We fix an element $e_b^*\in E$ in each fibre $p^{-1}(b)$ for $b\in B$ and set 
%\begin{center}
%	$V_i=\{(e_1,  \cdots, e_r) \in E_B^r ~|~ (e_1,  \cdots, e_r, e_b^*)\in U_i \text{ for some } b\in B\}.$
%\end{center}
%It is clear that $\{V_i:~ i=1,\cdots, k\}$ covers $E_B^r$. For any 	$(e_1, \cdots, e_r)\in V_i$ we define $s'_i : V_i \to E_B^{I_{r}}$ as $s'_i(e_1, \cdots, e_r)(t_j) = s_i(e_1, \cdots, e_r, e_b^*)(t_j)$ for $t_j\in [0, 1]_j$ and $j=1, \cdots, r$. Then $s'_i$ is a section of $\Pi'_{r}: E_B^{I_{r}} \to E_B^{r}$. So $\TC_{r+1}[p: E \to B]\geq \TC_{r}[p: E \to B]$, for any $r\geq 2$.

%\end{section}
\section{Upper and lower bounds for $\TC_r[p:E\to B]$} %sequential parametrized topological complexity}
% For any fibration $p : E \to B$, since it has homotopy lifting property, $\secat(p) \leq \cat(B)$. So for the fibration $\Pi_r : E_B^I \to E_B^r$ we can say that $\secat(\Pi_r)=\TC_r[p : E \to B]\leq \cat(E_B^r)$. Again we know that paracompact space $X$, $cat(X)\leq\dim(X)$ (cf. \cite{Jam78}). So one can say that $\TC_r[p : E \to B]\leq \dim(E_B^r)$. Now we consider the fibration $E_B^r \to B$ given by the map $p$.

In this section we give upper and lower bound for sequential parametrized topological complexity.

\begin{prop}\label{prop upper bound}
	Let $p: E \to B$ be a locally trivial fibration with fiber $X$, where $E, B, X$ are CW-complexes. Assume that the fiber $X$ is $k$-connected, where $k\ge 0$. Then 
	\begin{eqnarray}\label{upper}
	\TC_{r}[p: E \to B]<\frac{\hdim (E_B^r) + 1}{k+1}\leq \frac{r\cdot \dim X+\dim B + 1}{k+1}.
	\end{eqnarray}
\end{prop}

\begin{proof}
Since $X$ is $k$-connected, the loop space $\Omega X$ is $(k-1)$-connected and hence the space $(\Omega X)^{r-1}$ is also $(k-1)$-connected. Thus, the fibre of the fibration $\Pi_r : E_B^I \to E_B^r$ is $(k-1)$-connected and applying Theorem 5 from \cite{Sva66} we obtain:
%Using \cite[Theorem 5]{Sva66} we have 
\begin{equation}\label{equation upper bound}
\TC_{r}[p: E \to B] \, =\, \secat(\Pi_r)\, < \, \frac{\hdim (E_B^r) + 1}{k+1},
\end{equation} where $\hdim(E_B^r)$ denotes  homotopical dimension of $E_B^r$, i.e. the minimal dimension of a CW-complex homotopy equivalent to $E^r_B$, 
$$\hdim(E_B^r ) := \min\{\dim \, Z | \, Z \, \mbox{is a CW-complex homotopy equivalent to} \, E_B^r \}.$$
Clearly, $\hdim(E_B^r)\leq \dim(E_B^r)$. 
The space $E^r_B$ is the total space of a locally trivial fibration with base $B$ and fibre $X^r$. 
%We consider the fibration $E_B^r \to B$ with fibre $X^r$ given by the map $p$. 
Hence, $\dim(E_B^r)\, \leq \, \dim(X^r)+ \dim B= r\cdot \dim X+ \dim B$. Combining this with  (\ref{equation upper bound}), 
we obtain (\ref{upper}). 
%$$\TC_{r}[p: E \to B]<\frac{\hdim (E_B^r) + 1}{k+1}\leq \frac{r\cdot \dim X+\dim B + 1}{k+1}.$$
\end{proof}

Below we shall use the following classical result of A.S. Schwarz \cite{Sva66}: 
\begin{lemma}\label{lemma lower bound of secat}	
For any fibration $p: E \to B$ and coefficient ring $R$, if there exist cohomology classes $u_1, \cdots, u_k\in \ker[p^*: H^*(B; R)\to H^*(E; R)]$ such that their cup-product is nonzero, $u_1 \cup \cdots \cup u_k \neq 0 \in H^*(B; R)$, then $\secat(p) \geq k$.
\end{lemma}
	
	The following Proposition gives a simple and powerful lower bound for sequential parametrized topological complexity. 
	
\begin{prop}\label{lemma lower bound for para tc}
For a fibration $p: E \to B$, consider the diagonal map $\Delta : E \to E^r_B$ where $\Delta(e)= (e, e, \cdots, e)$, and the induced by $\Delta$ homomorphism in cohomology $\Delta^\ast: H^\ast(E^r_B;R) \to H^\ast(E;R)$ with coefficients in a ring $R$. 
If there exist cohomology classes $$u_1, \cdots, u_k \in \ker[\Delta^* : H^*(E_B^r; R) \to H^*(E; R)]$$ 
such that 
$$u_1 \cup \cdots \cup u_k \neq 0 \in H^*(E_B^r; R)$$ then $\TC_{r}[p: E \to B]\geq k$. 
\end{prop}
\begin{proof}
Define the map $c : E \to E_B^I$ where $c(e)(t) = e$ is the constant path. Note that the map $c : E \to E_B^I$ is a homotopy equivalence. Besides, the following diagram commutes 
	$$
	\xymatrix{
		E \ar[rr]^{c} \ar[dr]_{\Delta}&&E_B^I \ar[dl]^{\Pi_r}\\
		& E_B^r & &
	}
	$$
and thus, $\ker[\Pi_r^*: H^*(E_B^r; R) \to H^*(E_B^I; R)] = \ker[\Delta^*: H^*(E_B^r; R) \to H^*(E; R)]$. The result now follows from Lemma \ref{lemma lower bound of secat} and from the definition $\TC_r[p:E\to B]=\secat(\Pi_r)$. 
\end{proof}

\section{Cohomology algebras of certain configuration spaces}

In this section we present auxiliary results about cohomology algebras of relevant configuration spaces which will be used later in this paper for computing the sequential parametrized topological complexity of the Fadell - Neuwirth fibration. 

All cohomology groups will be understood as having the integers as coefficients although the symbol $\Bbb Z$ will be skipped from the notations.

We start with the following well-known fact, see \cite[Chapter V, Theorem 4.2]{FadH01}:

\begin{lemma}\label{lemma cohomology of configuration space}
The integral cohomology ring $H^\ast(F(\rr^d, m+n))$ contains $(d-1)$-dimensional cohomology classes 
$\omega_{ij}$, where $1\leq i< j\leq m+n,$ 
which multiplicatively generate $H^*(F(\rr^d, m+n))$ and satisfy the following defining relations 
$$(\omega_{ij})^2=0 \quad \mbox{and}\quad 
 \omega_{ip}\omega_{jp}= \omega_{ij}(\omega_{jp}-\omega_{ip})\quad \text{for all }\,  i<j<p.$$
\end{lemma}

The cohomology class $\omega_{ij}$ arises as follows. For $1\le i<j\le m+n$, mapping a configuration 
$(u_1, \dots, u_{m+n})\in F(\rr^d, m+n)$ to the unit vector 
$$\frac{u_i-u_j}{||u_i-u_j||}\, \in\,  S^{d-1},$$
defines a continuous map 
$\phi_{ij}: F(\rr^d, m+n)\to S^{d-1}$, and the class $$\omega_{ij}\in H^{d-1}(F(\rr^d, m+n))$$ is defined by 
$\omega_{ij}=\phi^\ast_{ij}(v)$
where $v\in H^{d-1}(S^{d-1})$ is the fundamental class. 

%The space 
%
%\begin{proof} 
%Note that $F(\rr^d, 2)$ is homotopy equivalent to $d-1$ dimensional sphere $S^{d-1}$. Consider the projection maps $p_{ij}: E=F(\rr^d, m+n) \to F(\rr^d, 2),$ where $i<j,$ to $i$-th and $j$-th coordinate. We denote $p^*_{ij}(\alpha)=\omega_{ij}$ where $\alpha$ is fundamental cohomology class in $H^{d-1}(F(\rr^d, 2))=H^{d-1}(S^{d-1})$. The required relation satisfies by \cite[Chapter V, Theorem 4.2]{FadH01}, as each cohomolgy class $\omega_{ij}$ is pullback of cohomology class of configuration space $F(\rr^d, 2)$.\end{proof}

Below we shall denote by $E$ the configuration space $E=F(\rr^d, n+m)$. A point of $E$ will be understood as a configuration 
$$(o_1, o_2, \dots, o_m, z_1, z_2, \dots, z_n)$$
where the first $m$ points $o_1, o_2, \dots, o_m$ represent \lq\lq obstacles\rq\rq\,  while the last $n$ points $z_1, z_2, \dots, z_n$ represent \lq\lq robots\rq\rq. The map 
\begin{eqnarray}\label{FN}
p: F(\rr^d, m+n) \to F(\rr^d, m),
\end{eqnarray}
where
$$ p(o_1, o_2, \dots, o_m, z_1, z_2, \dots, z_n)= (o_1, o_2, \dots, o_m),$$
is known as the Fadell - Neuwirth fibration. This map was introduced in \cite{FadN} where the authors showed that $p$ is a locally trivial fibration. The fibre of $p$ over a configuration $\mathcal O_m=\{o_1, \dots, o_m\}\in F(\rr^d, m)$
is the space $X=F(\rr^d- \mathcal O_m, n)$, the configuration space of $n$ pairwise distinct points lying in the complement of the set $\mathcal O_m=\{o_1, \dots, o_m\}$ of $m$ fixed obstacles.

We plan to use Lemma \ref{lemma lower bound for para tc} to obtain lower bounds for the topological complexity, and for this reason our first task will be to calculate the integral cohomology ring of the space 
$E_B^r$. Here $E$ denotes the space $E=F(\rr^d, m+n)$ and $B$ denotes the space $B=F(\rr^d, m)$ and $p: E\to B$ is the Fadell - Neuwirth fibration (\ref{FN}); hence $E^r_B$ is the space of $r$-tuples $(e_1, e_2, \dots, e_r)\in E^r$ satisfying 
$p(e_1)=p(e_2)=\dots= p(e_r)$. 
Explicitly, a point of the space $E_B^r$ can be viewed as a configuration
\begin{eqnarray}\label{confr}
(o_1, o_2, \cdots o_m, z^1_1, z^1_2, \cdots z^1_n, z^2_1, z^2_2, \cdots, z^2_n, \cdots, z^r_1, z^r_2, \cdots, z^r_n)
\end{eqnarray}
of $m+rn$ points %of $\rr^d$ where the points
$o_i, \, z^l_j\in \rr^d$ (for 
$i=1, 2, \dots, m$, $j=1, 2, \dots, n$ and $l=1, 2, \dots, r$), 
%$1\leq i \leq m, 1\leq j \leq n$ and $1\leq l \leq r$ 
such that
\begin{enumerate}
\item $o_i \neq o_{i'}$ for $i\neq i'$,
\item $o_i \neq z^l_j$ for $1\leq i \leq m, \, 1\leq j \leq n$ and $1\leq l \leq r$,
\item 
$z^l_j \neq z^l_{j'}$ for $j \neq j'$.
\end{enumerate}

The following statement is a generalisation of Proposition 9.2 from \cite{CohFW21}.
\begin{prop}\label{prop relation of cohomology classes}
The integral cohomology ring $H^*(E_B^r)$ contains cohomology classes 
$\omega^l_{ij}$ 
of degree $d-1$, where $1\leq i < j \leq m+n$ and $1\leq l \leq r$,
satisfying the relations 
\begin{enumerate}
\item[{\rm (a)}] $\omega_{ij}^l = \omega_{ij}^{l'} \text{ for } 1\leq i < j \leq m$  and $1\leq l \leq l' \leq r$,
\item[{\rm (b)}] $(\omega_{ij}^l)^2=0\text{ for } i < j \text{ and } 1\leq l \leq r$,
\item[{\rm (c)}] $\omega_{ip}^l\omega_{jp}^l= \omega_{ij}^l(\omega_{jp}^l-\omega_{ip}^l) \text{ for } i<j<p \text { and }\,  1\leq l \leq r.$
\end{enumerate}

\end{prop}

\begin{proof} For $1\le l\le r$, 
 consider the 
 projection map $q_l : E_B^r \to E$ which acts as follows: the configuration (\ref{confr}) is mapped into 
 $$(u_1, u_2, \dots, u_{m+n})\in E=F(\rr^d, m+n)$$
 where 
 $$
 u_i=\left\{
 \begin{array}{lll}
 o_i& \mbox{for} & i\le m,\\ 
 z^l_{i-m}& \mbox{for} & i> m.
 \end{array}
 \right.
 $$
 Using Lemma \ref{lemma cohomology of configuration space} and the cohomology classes 
 $\omega_{ij}\in H^{d-1}(E)$, we define $$(q_{l})^*(\omega_{ij})=\omega^l_{ij}\, \in \, H^*(E_B^r).$$
 Relations {\rm {(a), \, (b), \, (c)}} are obviously satisfied. This completes the proof. 
\end{proof}

For $1\leq i < j \leq m$ we shall denote the class $\omega_{ij}^l\in H^{d-1}(E^r_B)$ simply by $\omega_{ij}$; this is justified because of the relation {\rm {(a)}} above. 

We shall introduce notations for the classes which arise as the cup-products of the classes $\omega^l_{ij}$. For $p\ge 1$
consider two sequences of integers 
$I=(i_1, i_2, \cdots, i_p)$ and $J=(j_1, j_2, \cdots, j_p)$ where $i_s, j_s\in \{1, 2, \dots, m+n\}$ for $s=1, 2, \dots, p$.  
We shall say that the sequence $J$ is {\it increasing} if either $p=1$ or $j_1<j_2<\dots <j_p$. Besides, we shall write $I<J$ if $i_s<j_s$ for all $s=1, 2, \dots, p$. 

A pair of sequences $I<J$ of length $p$ as above determines the cohomology class
$$\omega_{IJ}^l=\omega^l_{i_1j_1}\omega^l_{i_2j_2}\cdots \omega^l_{i_pj_p}\in H^{(d-1)p}(E_B^r)$$ 
for any $l=1, 2, \dots, r$. Note that the order of the factors is important in the case when the dimension $d$ is even. 
Because of the property {\rm {(a)}} of Proposition \ref{prop relation of cohomology classes}
this class is independent of $l$ assuming that $j_p\le m$; for this reason, if $j_p\le m$, 
we shall denote $\omega^l_{IJ}$ simply by $\omega_{IJ}$.

For formal reasons we shall allow $p=0$. In this case the symbol $\omega^l_{IJ}$ will denote the unit $1\in H^0(E^r_B)$. 

The next result is a generalisation of \cite[Proposition 9.3]{CohFW21} where the case $r=2$ was studied. 
%
%\vskip 1cm
%
% $\omega_{ij}^l = \omega_{ij}^{l'}$   we shall denote the class $\omega_{ij}^l $
%
%, so for each $l$ we may denote the $\omega^l_{ij}$ by $\omega_{ij}$  for  $1\leq i < j \leq m$. Now we generalised the \cite[Proposition 9.3]{CohFW21}. 
%
%Consider the sequences $I=(i_1, i_2, \cdots, i_p)$ and $J=(j_1, j_2, \cdots, j_p)$. 
%
%Here we assume that $1\leq i_s<j_s\leq m+n$ for each $s$. In this case we write $I<J$. We denote $$\omega_{IJ}=\omega_{i_1j_1}\omega_{i_2j_2}\cdots \omega_{i_pj_p}\in H^{(d-1)p}(E_B^r)$$
%when  $J$ takes valus in $\{2, 3, \cdots, m\}$, and for $1\leq l \leq r$, 
%
%$$\omega_{IJ}^l=\omega^l_{i_1j_1}\omega^l_{i_2j_2}\cdots \omega^l_{i_pj_p}\in H^{(d-1)p}(E_B^r)$$ 
%
%when $J$ takes valus in $\{m+1, m+2, \cdots, m+n\}$.
%
%\vskip 1cm

\begin{prop}\label{prop basic cohomology classes}
An additive basis of $ H^*(E_B^r)$ is formed by the following set of cohomology classes 
\begin{eqnarray}\label{basis}
\omega_{IJ}\omega^1_{I_1J_1}\omega^2_{I_2J_2}\cdots \omega^r_{I_rJ_r}\, \in\,  H^\ast(E^r_B)\quad \mbox{with}\quad I<J\quad \mbox{and}\quad I_i<J_i,\end{eqnarray}
where: 
\begin{enumerate}
\item[{\rm (i)}] the sequences $J, J_1, J_2, \cdots J_r$ are increasing,
\item[{\rm (ii)}] the sequence $J$ takes values in $\{2, 3, \cdots, m\}$,
\item[{\rm (iii)}] the sequences $J_1, J_2, \dots, J_r$ take values in $\{m+1, \cdots m+n\}$.
\end{enumerate}
\end{prop}
\begin{proof} Recall our notations: $E=F(\rr^d, m+n)$, \,  $B=F(\rr^d, m)$ and $p: E\to B$ is the Fadell - Neuwirth fibration
(\ref{FN}). Consider the fibration 
$$p_r: E^r_B \to B \quad \mbox{where}\quad p_r(e_1, \dots, e_r) =p(e_1)=\dots=p(e_r).$$ 
Its fibre over a configuration $\mathcal O_m=(o_1, \dots, o_m)\in B$ 
is $X^r$, the Cartesian product of $r$ copies of the space $X$, where
$X=F(\rr^d-\mathcal O_m, n)$. 

We shall apply Leray-Hirsch theorem to the fibration $p_r: E^r_B \to B$. The classes $\omega_{ij}$
with $i < j \leq  m$ originate from the base of this fibration. Moreover, from Lemma \ref{lemma cohomology of configuration space} it follows that a free additive basis of $H^\ast(B)$ forms the set of the classes $\omega_{IJ}$ where $I<J$ run over all sequences of elements of the set  $\{ 1,2,...,m\} $ such that the sequence $J = (j_1,j_2,...,j_p)$ is increasing.

Next consider the classes of the form
$$\omega^1_{I_1J_1}\omega^2_{I_2J_2}\cdots \omega^r_{I_rJ_r}\, \in\,  H^\ast(E^r_B),$$
with increasing sequences $J_1, J_2, \dots, J_r$ satisfying {\rm {(iii)}} above. 
Using the known
results about the cohomology algebras of configuration spaces (see \cite{FadH01}, Chapter V, Theorems
4.2 and 4.3) as well as the Ku\"nneth theorem, we see that the restrictions of the family of these 
classes onto the fiber $X^r$ form a free basis in the cohomology of the fiber $H^\ast (X^r).$

Hence, Leray-Hirsch theorem \cite{Hat} is applicable and we obtain that a free basis of the cohomology 
$H^\ast (E^r_B)$ is given by the set of classes described in the statement of Proposition \ref{prop basic cohomology classes}.
This completes the proof. 
\end{proof}

Proposition \ref{prop basic cohomology classes} implies:

\begin{corollary}\label{new}
Consider two basis elements $\alpha, \beta \in H^*(E_B^r)$
$$
\alpha = \omega_{IJ}\omega^1_{I_1J_1}\omega^2_{I_2J_2}\cdots \omega^r_{I_rJ_r}\,\quad \mbox{and}\quad
\beta = \omega_{I'J'}\omega^1_{I'_1J'_1}\omega^2_{I'_2J'_2}\cdots \omega^r_{I'_rJ'_r},
$$ 
satisfying the properties {\rm (i), (ii), (iii)} of Proposition \ref{prop basic cohomology classes}. 
The product $$\alpha\cdot \beta\in H^\ast(E^r_B)$$ is another basis element up to sign (and hence is nonzero) if 
the sequences $J$ and $J'$ are disjoint and for every $k=1, 2, \dots,r$ the sequences $J_k$ and $J'_k$ are disjoint. 
\end{corollary}

There is a one-to-one correspondence between increasing sequences and subsets; this explains the meaning of the term "disjoint" applied to two increasing sequences.

Next we consider the situation when the product of basis elements is not a basis element but rather a linear combination of basis elements. 

Let $J=(j_1, j_2, \dots, j_p)$ be an increasing sequence of positive integers, where $p\ge 2$,
and let  $j$ be an integer satisfying $j_p<j$. Our goal is to represent the product
$$
\omega^l_{j_1 j}\omega^l_{j_2 j}\dots \omega^l_{j_p j} \in H^{p(d-1)}(E^r_B), \quad l=1, \dots, r,
$$
as a linear combination of the basis elements of Proposition \ref{prop basic cohomology classes}. 

We say that a sequence 
$I=(i_1, i_2, \dots, i_p)$ with $p\ge 2$ is a {\it $J$-modification} if $i_1=j_1$ and for $s=2, 3, \dots, p$ each number $i_s$  equals either $i_{s-1}$
or $j_s$. An increasing sequence of length $p$ has $2^{p-1}$ modifications. 
For example, for $p=3$ the sequence 
$J= (j_1, j_2, j_3)$ has the following $4$ modifications
\begin{eqnarray}\label{mod}
(j_1, j_2, j_3), \, \, (j_1, j_1, j_3), \, \, (j_1, j_2, j_2), \, \, (j_1, j_1, j_1).
\end{eqnarray}

For a $J$-modification $I$ we shall denote by $r(I)$ the number of repetitions in $I$. For instance, the numbers of repetitions of the modifications (\ref{mod}) are $0, \, 1, \, 1, \, 2$ correspondingly. 

The following statement is equivalent to Proposition 3.5 from \cite{CohFW}. Lemma 9.5 from \cite{CohFW21} 
gives the answer in a recurrent form. 

\begin{lemma}\label{lm:mod} For a sequence  $j_1<j_2<\dots < j_p<j$ of positive integers, where $p\ge 2$, denote 
$J=(j_1, j_2, \dots, j_p)$ and $J'=(j_2, j_3, \dots, j_p, j)$. 
In the cohomology algebra $H^\ast(E^r_B)$ associated to the Fadell - Neuwirth fibration, 
one has the following %$2^{p-1}$-term
relation
\begin{eqnarray}\label{dec}
\omega^l_{j_1 j}\omega^l_{j_2 j}\dots \omega^l_{j_p j} = \sum_I (-1)^{r(I)} \omega^l_{I J'},
\end{eqnarray} 
where $I$ runs over $2^{p-1}$ $J$-modifications and $l=1, 2, \dots, r.$
\end{lemma}
\begin{proof} First note that for any $J$-modification $I$ one has $I< J'$ and hence the terms in the RHS of (\ref{dec}) make sense. We shall use induction in $p$. For $p=2$ the statement of Lemma \ref{lm:mod} is 
$$
\omega^l_{j_1j}\omega^l_{j_2j}= \omega^l_{j_1j_2}\omega^l_{j_2j}-\omega^l_{j_1j_2}\omega^l_{j_1 j},
$$
which is the familiar $3$-term relation, see Proposition \ref{prop relation of cohomology classes}, statement (c).
The first term on the right corresponds to the sequence $I=(j_1,j_2)$ and the second term corresponds to 
$I=(j_1, j_1)$; the latter has one repetition and appears with the minus sign. 

Suppose now that Lemma \ref{lm:mod} is true for all sequences $J$ of length $p$. 
Consider an increasing sequence $J=(j_1, j_2, \dots, j_{p+1})$ of length $p+1$ and 
an integer $j$ satisfying $j>j_{p+1}$. Denote by $K=(j_1, j_2, \dots, j_p)$ the shortened sequence and let 
$I=(i_1, i_2, \dots, i_p)$ be a modification of $K$. As in (\ref{dec}), denote $K'=(j_2, j_3, \dots, j_p, j)$.
Consider the product
\begin{eqnarray*}
\omega^l_{I K'}\omega^l_{j_{p+1}j}&=&
\omega^l_{i_1j_2}\omega^l_{i_2j_3}\dots \omega^l_{i_{p-1}j_p} \omega^l_{i_p j}\cdot \omega^l_{j_{p+1}j}\\
&=& \left[\omega^l_{i_1j_2}\omega^l_{i_2j_3}\dots \omega^l_{i_{p-1}j_p} 
\right]\cdot \omega^l_{i_pj_{p+1}}\cdot \left[\omega^l_{j_{p+1}j} -\omega^l_{i_pj}\right]\\
&=& \omega^l_{I_1J'}-\omega^l_{I_2J'}
\end{eqnarray*}
where $I_1=(i_1, \dots, i_p, j_{p+1})$ and $I_2= (i_1, \dots, i_p, i_p)$  are the only two modifications of $J$ extending $I$. 
The equality of the second line is obtained by applying the relation (c) of Proposition \ref{prop relation of cohomology classes}.
Note that $r(I_1)=r(I)$ and $r(I_2)=r(I_1)+1$ which is consistent with the minus sign. 
Thus, we see that Lemma follows by induction. 
\end{proof}

Since each term in the RHS of (\ref{dec}) is a $\pm$ multiple of a basis element we obtain: 
 
\begin{corollary}\label{cor:form}
Any basis element (\ref{basis}) which appears with nonzero coefficient in the decomposition of the monomial  
\begin{eqnarray}\label{prodd}
\omega^l_{j_1 j}\omega^l_{j_2 j}\dots \omega^l_{j_p j}, \quad \mbox{where}\quad j_1<j_2<\dots < j_p< j,
\end{eqnarray}
contains a factor of the form $\omega^l_{j_s j}$, where $s\in \{1, 2, \dots, p\}$. Moreover, 
$$\omega^l_{j_1 j_2}\omega^l_{j_1 j_3}\dots \omega^l_{j_1 j_p}\omega^l_{j_1 j}$$ is the only basis element in the decomposition of (\ref{prodd}) which contains the factor 
$\omega^l_{j_1 j}$. 
\end{corollary}

Consider the diagonal map 
$$\Delta : E \to E_B^r, \quad \Delta(e) = (e, e, \dots, e), \quad e\in E.$$ 
\begin{lemma} \label{lm:ker}
The kernel of the homomorphism $\Delta^* : H^*(E_B^r) \to H^*(E)$ contains
the cohomology classes of the form 
$$\omega^l_{ij}-\omega^{l'}_{ij}.$$ 
\end{lemma}

\begin{proof} This follows directly from the definition of the classes $\omega^l_{ij}$; compare the proof of Proposition 9.4 from
\cite{CohFW21}. 
%Since the map $\Delta$ is defined by $$\Delta(o_1, \cdots o_m, z_1, \cdots z_n)=(o_1, \cdots o_m, z_1, \cdots z_n, z_1, \cdots, z_n, \cdots, z_1, \cdots, z_n),$$ where the last $n$ points are repeated $r$-times. Now the lemma follows from the definition of the classes $\omega^l_{ij}$ and $\omega^{l'}_{ij}$ (see proof of Proposition \ref{prop relation of cohomology classes}).

\end{proof}
  %For any $p\geq0$ we denote $\omega_{e, p}= \omega_{e_{i_1}}\omega_{e_{i_2}} \cdots \omega_{e_{i_p}}\in H^{(d-1)p}(E_B^r)$ is the cup product of cohomology classes. We have the following proposition which is higher analogue of the Proposition 9.3 of \cite{CohFW21}. \begin{prop} A free basis of $H^*(E_B^r)$ is given by the set of cohomology classes of the form $$\omega_{e^0, p_0}\omega_{e^1, p_1}\cdots \omega_{e^{r}, p_{r}}$$ where all edges corresponds to the class $\omega_{e^0, p_0}$ is in $K_m$ and the edges corresponds to the class  $\omega_{e^l, p_l}$ is between $K_m$ and $K^l_n$ for $1\leq l \leq r$.  \end{prop}

\section{Sequential parametrized topological complexity of the Fadell-Neuwirth bundle; the odd-dimensional case}\label{sec:odd}

%Consider the configuration space $F(\rr^d, m+n)$ of $m+n$ pairwise distinct points in $\rr^d$. A configuration $c \in F(\rr^d, m+n)$
%will be viewed as a collection $$c=(o_1, o_2, \cdots, o_m, z_1, z_2, \dots, z_n)$$ where the points $o_1, \dots, o_m\in \rr^d$ represent positions of the obstacles and the points $z_1, \dots, z_n\in \rr^d$ represent positions of the robots. The map 
%\begin{eqnarray}\label{FN}
%p: F(\rr^d, m+n) \to F(\rr^d, m),
%\end{eqnarray}
%where
%$$ p(o_1, o_2, \cdots, o_m, z_1, z_2, \dots, z_n)= (o_1, o_2, \cdots, o_m),$$
%is known as the Fadell - Neuwirth fibration. This map was introduced in \cite{FadN} where the authors showed that $p$ is a locally trivial fibration. The fibre of $p$ over a configuration $\mathcal O_m=\{o_1, \dots, o_m\}\in F(\rr^d, m)$
%is the space $X=F(\rr^d- \mathcal O_m, n)$, the configuration space of $n$ pairwise distinct points lying in the complement of a set $\mathcal O_m=\{o_1, \dots, o_m\}$ of $m$ fixed obstacles. 

Our goal is to compute the sequential parametrized topological complexity of the Fadell - Neuwirth bundle. 
As we shall see, the answers in the cases of odd and even dimension $d$ are slightly different. When $d$ is odd the cohomology algebra has only classes of even degree and is therefore commutative; in the case when $d$ is even the cohomology algebra is skew-commutative which imposes major 
distinction in treating these two cases. 

The main result of this section is:

\begin{theorem}\label{thm:odd}
For any odd $d\ge 3$, and for any $n\ge 1$, $m\ge 2$ and $r\ge 2$, the sequential parametrized topological complexity of the Fadell - Neuwirth bundle (\ref{FN}) equals 
$rn+m -1$. 
\end{theorem}

This result was obtained in \cite{CohFW21} for $r=2$. 
Note that the special case of $d=3$ is most important for robotics applications. 

As in the previous section, we shall denote the Fadell - Neuwirth bundle (\ref{FN}) by $p: E\to B$ for short; this convention will be in force in this and in the following sections. 

%Many statements of this section are valid for any dimension $d$, regardless whether it is even or odd. 

We start with the following statement which is valid without imposing restriction on the parity of the dimension $d\ge 3$. 
Note that for $d= 2$ we shall have a stronger upper bound in \S \ref{sec:9}. 

\begin{prop}\label{prop upper bound for Fadell-Neuwirth bundle}
For any $d\geq 3$ and $m\ge 2$ one has $$\TC_r[p:  E\to B] \leq rn+m-1.$$ 
\end{prop}

\begin{proof} The space $E^r_B$ is $(d-2)$-connected and in particular it is
simply connected (since $d\ge 3)$. By Proposition \ref{prop basic cohomology classes} the top dimension with nonzero cohomology is $(rn + m-1)(d-1)$. Hence the 
 homotopical dimension of the configuration space $\hdim (E^r_B)$ equals 
$(rn+m-1)(d-1)$. Here we use the well-known fact that the homotopical dimension of a simply connected space with torsion free integral cohomology equals its cohomological dimension. The fibre $X=F(\rr^d - \mathcal O_m, n)$ of the Fadell - Neuwirth bundle 
$p: E\to B$ is $(d-2)$-connected. Applying Proposition \ref{prop upper bound} we obtain 
$$\TC_r[p : E\to B]\, <\, rn+m-1+\frac{1}{d-1},$$
which is equivalent to our statement. 
\end{proof}

To complete the proof of Theorem \ref{thm:odd} we only need to establish the lower bound:

\begin{prop}\label{odd lower bound for Fadell-Neuwirth bundle}
For any odd $d\geq 3$ and $m\ge 2$ one has $$\TC_r[p:  E\to B] \geq rn+m-1.$$
\end{prop}
Note that the assumption of this Proposition that the dimension $d$ is odd is essential as Proposition \ref{odd lower bound for Fadell-Neuwirth bundle} is false for $d$ even, see below. 

\begin{proof} We shall use Lemma \ref{lemma lower bound for para tc} and Propositions \ref{prop relation of cohomology classes}
and \ref{prop basic cohomology classes} and Lemma \ref{lm:ker}. 

Consider the cohomology classes
\begin{eqnarray*}
x_1& =& \prod_{i=2}^m(\omega^1_{i(m+1)} - \omega^2_{i(m+1)})\, \in \, H^{(m-1)(d-1)}(E^r_B),\\
x_2 &=& \prod_{j=m+1}^{m+n}(\omega_{1j}^{2} - \omega_{1j}^{1})^2 \, =\, -2 \prod_{j=m+1}^{m+n} \omega^1_{1j}\omega^2_{1j}\, \in \, H^{2n(d-1)}(E^r_B),\\
x_3&=& \prod_{l=3}^{r}\prod_{j=m+1}^{m+n}(\omega_{1j}^{l} - \omega_{1j}^{1})\, \in H^{n(r-2)(d-1)}(E^r_B).
\end{eqnarray*}
Each of these classes is a product of elements of the kernel of the homomorphism $\Delta^\ast: H^\ast(E^r_B)\to H^\ast(E)$, by Lemma \ref{lm:ker}. Proposition \ref{odd lower bound for Fadell-Neuwirth bundle} would follow once we show that 
the product 
\begin{eqnarray*}
x_1 x_2 x_3 \not= 0\, \in \, H^\ast(E^r_B)
\end{eqnarray*}
is nonzero. By Proposition \ref{prop basic cohomology classes}, the product $x_1 x_2 x_3$ is a linear combination of the basis cohomology classes and it is 
nonzero if at least one coefficient in this decomposition does not vanish. 

According to \cite{CohFW21}, cf. page 248, the product $x_1 x_2$ contains the basis element
\begin{eqnarray}
\omega_{I_0J_0} \omega^1_{IJ} \omega^2_{I'J}\, \in \, H^{(2n+m-1)(d-1)}(E^r_B)
\end{eqnarray}
with a nonzero coefficient; here 
\begin{eqnarray*}
I_0&=& (1, 2, 2, \dots, 2), \quad
J_0= (2, 3, \dots, m), 
\end{eqnarray*}
and 
\begin{eqnarray*}
I= (1, 1, \dots, 1), \quad 
I'= (2, 1, 1, \dots, 1),\quad 
J= (m+1, m+2, \dots, m+n), 
\end{eqnarray*}
with $|I_0|= |J_0|=m-1$ and $ |I|=|I'|= |J|=n.$

The product representing $x_3$ can be expanded into a sum. This sum contains the class $\prod_{l=3}^{r}\omega^l_{IJ}$ 
and each of the other terms contains a factor of type $\omega^1_{1j}$. Since obviously $x_1x_2\omega^1_{1j}=0,$ we obtain that the product $x_1x_2x_3$ contains the basis element 
$$
\omega_{I_0J_0}\cdot \omega^1_{IJ}\cdot \omega^2_{I'J}\cdot \prod_{l=3}^{r}\omega^l_{IJ}
$$
with a nonzero coefficient. Hence $x_1 x_2x_3 \not=0$ is nonzero. This completes the proof of Proposition \ref{odd lower bound for Fadell-Neuwirth bundle}.
\end{proof}

\begin{remark}
The lower bound estimate of Proposition \ref{odd lower bound for Fadell-Neuwirth bundle} fails to work in the case when the dimension $d$ of the ambient Euclidean space is even. Indeed, then the classes $\omega^l_{ij}$ have odd degree (which equals $d-1$) 
and the square of any class of odd degree vanishes (since the cohomology algebra $H^\ast(E^r_B)$ with integral coefficients is torsion free). Thus, in the case of even dimension $d$ the product $x_2$ vanishes. In the following section
we shall suggest a different estimate for $d$ even . 
\end{remark}

\section{Sequential parametrized topological complexity of the Fadell-Neuwirth bundle; the even-dimensional case}\label{sec:even}\label{sec:9}

In this section we give a lower bound for $\TC_r[p:E\to B]$ for the Fadell - Neuwirth bundle (\ref{FN}) 
in the case when the dimension $d$ of the Euclidean space $\rr^d$ is even. 
We also prove a matching upper bound for the planar case $d=2$. 
Such an upper bound can be obtained for any even $d$ by a totally different method; this material will be presented in another publication.

First we establish the following lower bound which is valid for any $d$ regardless of its parity.  

\begin{prop}\label{prop lower bound for Fadell-Neuwirth bundle even d}
For any $d\ge 2$, $r\ge 2$ and $m\ge 2$,
the sequential parametrized topological complexity of the Fadell - Neuwirth bundle 
satisfies
\begin{eqnarray}\label{lowereven}
\TC_r[p:E\to B]\ge rn+m-2.
\end{eqnarray} 
\end{prop}

\begin{proof}
As an illustration, consider first the special case $m=2$ and $n=1$, i.e. the situation when we have one robot and two obstacles.  
Then the product of $r$ classes 
\begin{eqnarray}\label{15}
(\omega^1_{23}-\omega^2_{23}) \cdot \prod_{l=2}^r (\omega^l_{13}-\omega^1_{13})
\end{eqnarray}
lying in the kernel of $\Delta^\ast$ 
contains the basis element 
\begin{eqnarray}\label{targetm2}
\omega^1_{23}\cdot \prod_{l=2}^r \omega^l_{13}\end{eqnarray}
with a nonzero coefficient. Indeed,  (\ref{15}) equals
$
(\omega^1_{23}-\omega^2_{23}) \cdot \left[\prod_{l=2}^r \omega^l_{13}-\omega^1_{13}\cdot \alpha\right]
$
where $\alpha$ is a polynomial in the classes $\omega^l_{13}$ with $l\in \{2, \dots, r\}$.
Opening the brackets gives 
$$
\omega^1_{23}\cdot \prod_{l=2}^r \omega^l_{13} - \omega^2_{23} \cdot \prod_{l=2}^r \omega_{13}^l - \omega^1_{23}\omega^1_{13}\alpha + \omega^2_{23}\omega^1_{13}\alpha.
$$
Here the second and the third terms are the sums of basis elements each containing the factor $\omega_{12}$ and hence  distinct from 
(\ref{targetm2}). The basis elements of the fourth term all contain the factor $\omega^1_{13}$ and therefore are also distinct from 
(\ref{targetm2}). Thus, (\ref{15}) is nonzero and Lemma \ref{lemma lower bound for para tc} gives the desired lower bound in the case $m=2$, $n=1$.

Returning to the general case, consider the following three cohomology classes 
$x_1, x_2, x_3\in H^\ast(E^r_B)$, where
\begin{eqnarray*}
x_1&=& \prod_{i=2}^m(\omega^1_{i(m+1)} - \omega^2_{i(m+1)})\, \in H^{(m-1)(d-1)}(E^r_B),\\
x_2 &=& \prod_{j=m+2}^{m+n}(\omega_{(j-1)j}^{1} - \omega_{(j-1)j}^{2})\, \in H^{(n-1)(d-1)}(E^r_B), \\
x_3 &=& \prod_{l=2}^{r}\prod_{j=m+1}^{m+n}(\omega_{1j}^{l} - \omega_{1j}^{1}) \, \in H^{n(r-1)(d-1)}(E^r_B).
\end{eqnarray*}

Note that in the case when $n=1$ the class $x_2$ is not defined; however, the arguments below show that in the case $n=1$ the class
$x_1 x_3$ (which is the product of $r+m-2$ classes lying in the kernel of $\Delta^\ast$) is nonzero. 

Each of the classes $x_1, x_2, x_3$ is the product of elements of the kernel of $\Delta^\ast$, see Lemma \ref{lm:ker}, and the total number of the factors is $rn+m-2$. Hence, by Lemma \ref{lemma lower bound for para tc}, our statement (\ref{lowereven}) will follow 
once we know that the product $x_1x_2x_3\not=0\in H^\ast(E^r_B)$ is nonzero.  

Consider the following sequences
\begin{eqnarray*}
\begin{array}{lll}
I_0=(2, 2, \dots, 2), &\mbox{where} &|I_0|=m-2,\\
J_0=(3, 4, \dots, m),  &\mbox{where} &|J_0|=m-2,\\
I=(1, 1, \dots, 1), &\mbox{where} &|I|=n,\\
K=(2, m+1, m+2, \dots, m+n-1),  &\mbox{where} &|K|=n,\\
J=(m+1, m+2, \dots, m+n),  &\mbox{where}&|J|=n.
\end{array}
\end{eqnarray*}
We claim that the basis element
\begin{eqnarray}\label{target}
\omega_{I_0J_0}\omega^1_{KJ} \omega_{IJ}^2 \omega_{IJ}^3\dots \omega_{IJ}^r
\end{eqnarray}
appears in the decomposition of the product $x_1x_2x_3$ with a nonzero coefficient.

In the special case $n=1$ the class (\ref{target}) has the form
\begin{eqnarray}\label{target1}
\omega_{23}\omega_{24}\dots\omega_{2 m}\cdot \omega^1_{2 (m+1)}\cdot \prod_{l=2}^r \omega^l_{1 (m+1)}.
\end{eqnarray}

Consider the basis elements which appear in the decomposition of the class $x_1$. 
For $m=2$ the class $x_1$ equals $\omega^1_{23} - \omega^2_{23}$ and for $m>2$ 
we can write
\begin{eqnarray}\label{x1}
x_1= \sum_{R\subset [m]}\pm \, \left(\prod_{i\in R} \omega^1_{i (m+1)}\cdot \prod_{i\in R^c} \omega^2_{i (m+1)}\right),
\end{eqnarray}
where $R$ runs over all subsets (including $R= \emptyset$) of the set $[m]=\{2, 3, \dots, m\}$ and $R^c$ denotes the complement $[m]-R$. The terms of (\ref{x1}) are basis elements for $m=2$; for $m>2$ they can be decomposed into basis elements using Lemma \ref{lm:mod}. For example, taking $R=[m]$ and applying Lemma \ref{lm:mod} we find that one of the $2^{m-2}$ basis elements which appear in the decomposition of the product $\prod_{i=2}^m\omega^1_{i (m+1)}$ is the class
\begin{eqnarray}\label{class16}
\omega_{23}\omega_{24}\dots\omega_{2m}\omega^1_{2(m+1)}= \omega_{I_0J_0} \omega^1_{2(m+1)}.
\end{eqnarray}
This class clearly is a factor of (\ref{target}). For $m>2$ each other basis elements in the decomposition of $\prod_{i=2}^m\omega^1_{i (m+1)}$
has a factor of type $\omega^1_{i (m+1)}$ with $2<i\le m$, see Corollary \ref{cor:form}. 

Note also the basic elements of the form 
\begin{eqnarray}\label{16b}
\omega_{23}\omega_{24}\dots \omega_{2 (m-1)}\omega_{2 m}\omega^2_{k (m+1)}= \omega_{I_0J_0} \omega^2_{k (m+1)}, \quad \mbox{where}\quad 2\le k\le m,
\end{eqnarray}
which arise in the basic element decomposition of the summand of (\ref{x1}) with $R=\emptyset$. 

The basis element decomposition of $x_2$ is given by
\begin{eqnarray}\label{x2}
\sum_S \pm \left(\prod_{j\in S} \omega^1_{(j-1) j} \cdot \prod_{j\in S^c} \omega^2_{(j-1) j}\right),
\end{eqnarray}
where $S$ runs over all subsets $S\subset \{m+2, m+3, \dots, m+n\}$, including $S=\emptyset$. The symbol 
$S^c$ denotes the complement $\{m+2, m+3, \dots, m+n\}- S$. Taking $S= \{m+2, m+3, \dots, m+n\}$ in (\ref{x2}) gives the class
$\omega^1_{KJ}$, without the factor $\omega^1_{2 (m+1)}$, which is a factor of (\ref{target}). Note that the missing factor 
$\omega^1_{2 (m+1)}$ appears in (\ref{class16}). 

The basis element decomposition of the class $x_3$ is given by
\begin{eqnarray}\label{16}\label{x3}
\sum_{T_2, \dots, T_r} \pm \, \, \omega^1_{I_1 T_1}\omega^2_{I_2 T_2}\omega^3_{I_3 T_3} \dots \omega^r_{I_r T_r}, 
\end{eqnarray}
where $T_2, T_3, \dots, T_r$ run over subsets of the set $\{m+1, m+2, \dots, m+n\}$ 
such that every two of these sets cover $\{m+1, m+2, \dots, m+n\}$  and $T_1=\cup_{j=2}^r T_j^c$ where $T_j^c$ stands for the complement 
$\{m+1, m+2, \dots, m+n\} -T_j$. We identify the subsets of $\{m+1, m+2, \dots, m+n\}$  with increasing sequences in the obvious way. 
The sequences $I_1, I_2, \dots, I_r$ in (\ref{16}) all have the form $(1, 1, \dots, 1)$. Taking in (\ref{16}) $T_2=T_3=\dots=T_r=J$ gives the class $\omega^2_{IJ} \omega^3_{IJ} \dots \omega^r_{IJ}$ which is a factor 
of (\ref{target}). 

We have seen that the class (\ref{target}) appears as a product of specific basis elements in the decomposition of 
$x_1$, $x_2$ and $x_3$. We show below that the class (\ref{target}) appears {\it only} with the set of choices indicated above and hence it cannot be cancelled. 

Firstly, we note that only $x_3$ involves terms $\omega^l_{ij}$ with $l\ge 3$ and $j\ge m+1$. Therefore the only choice 
$T_3=T_4=\dots=T_r=J$ in (\ref{x3}) may possibly lead to (\ref{target}). 

Secondly, the basis elements in the decompositions of $x_2$ and $x_3$ have no factors $\omega^l_{ij}$ with $j\le m$. Hence the factor $\omega_{I_0J_0}$ of (\ref{target}) may only arise from the basis elements of the decomposition of $x_1$. It is clear that this may happen either when $R=[m]$ with (\ref{class16}) corresponding to the modification $(2, 2, \dots, 2)$ of the sequence
$(2, 3, \dots, m)$, or with $R=\emptyset$, see above. Any basis element of $x_1$ distinct from (\ref{class16}) 
has either a factor of type $\omega^1_{i (m+1)}$ with $3\le i\le m$ or a factor of type $\omega^2_{k (m+1)}$ with $2\le k\le m$. 
Such factors do not appear in (\ref{target}). If the set $T_1$ in (\ref{x3}) contains $m+1$ then we could have the factor 
$$\omega^1_{i (m+1)}\omega^1_{1 (m+1)}=\pm \omega_{1 i}(\omega^1_{i (m+1)} - \omega^1_{1 (m+1)})$$ with the factor 
$\omega_{1 i}$ missing in (\ref{target}). Similarly, the set $T_2$ might contain $m+1$ leading to the product 
$$\omega^2_{k (m+1)}\omega^2_{1 (m+1)}= \pm \omega_{1 k}(\omega^2_{k (m+1)} - \omega^2_{1 (m+1)})$$ with the factor 
$\omega_{1 k}$ being absent in (\ref{target}). 
Thus, we see that (\ref{class16}) is the only basis element of the decomposition of $x_1$ which can contribute into (\ref{target}). 

Comparing (\ref{x2}) and (\ref{target}) and using Corollary \ref{new} we see that the only basis element of the sum (\ref{x2}) with 
$S= \{m+2, m+3, \dots, m+n\}$ can contribute into (\ref{target}). This basis element, together with the factor 
$\omega^1_{2 (m+1)}$, gives $\omega^1_{KJ}$. 

Finally, examining (\ref{x3}), we see that the only way obtaining (\ref{target}) is by taking $T_2=J$ and hence $T_1=\emptyset$,
since, as we established earlier, one must have $T_3=\dots=T_r=J$ and $T_1=\cup_{j=2}^r T_j^c$.

Thus,  
the basis element (\ref{target}) appears in the decomposition of the product $x_1x_2x_3$ with a nonzero coefficient and hence $x_1x_2x_3\not=0.$ This completes the proof. 
\end{proof}

Next we state the main result of this section:

\begin{theorem}\label{thm:even}
For any $m\geq 2$, $n\ge 1$ and $r\ge 2$, the $r$-th sequential parametrized topological complexity of the Fadell-Neuwirth bundle in the plane is given by$$\TC_r[p : F(\rr^2, n+m) \to F(\rr^2, m)]=rn + m - 2.$$
\end{theorem}

\begin{proof}
Proposition \ref{prop lower bound for Fadell-Neuwirth bundle even d} gives the lower bound. In the proof below we
establish the upper bound. We shall adopt the method developed in  \cite{CohFW}. 
As in \cite{CohFW}, we identify $\rr^2$ with the set of complex numbers $\cc$ and for any $s\geq 3$ consider the 
homeomorphism 
$$h_s: F(\cc, s) \to F(\cc \smallsetminus \{0, 1\}, s-2)\times F(\cc, 2)$$ 
given by 
$$h_s(u_1, u_2, ..., u_s)=\left(\left(\frac{u_3-u_1}{u_2-u_1}, \frac{u_4-u_1}{u_2-u_1}, ..., \frac{u_s-u_1}{u_2-u_1} \right), 
(u_1, u_2)\right),$$ 
where $u_i\in \cc$, $u_i\not= u_j$ for $i\not=j$. Thus, using the algebraic structure of complex numbers we may split the configuration space into a product. 
We have the following commutative diagram 
$$\xymatrix{
F(\cc, n+m) \ar[d]_{p} \ar[r]^{\hskip -2 cm {h_{n+m}}} &  F(\cc \smallsetminus \{0, 1\}, n+m-2)\times F(\cc, 2) \ar[d]^{q\times \Id} \\
F(\cc, m) \ar[r]_{\hskip -2 cm {h_m}} & F(\cc \smallsetminus \{0, 1\}, m-2)\times F(\cc, 2)
}$$
where $p$ is the Fadell - Neuwirth fibration, $q$ is analogue of the Fadell - Neuwirth bundle for the plane with points $0, 1$ removed and with $m-2$ obstacles, and $\Id$ is the identity map. In the case when $m=2$ we shall consider the space 
$F(\cc \smallsetminus \{0, 1\}, m-2)$ as consisting of a single point; then the diagram above will make sense for $m=2$ 
(two obstacles only) as well.

Noting that $\TC_r[\Id : F(\cc, 2) \to F(\cc, 2)]=0$ and applying Proposition \ref{prop product inequality} we obtain 
\begin{equation*}
\label{equation para tc r=2}
\TC_r[p: F(\cc, n+m) \to F(\cc, m)]  \leq \TC_r[q: E' \to B'],
\end{equation*}
where $E'=F(\cc \smallsetminus \{0, 1\}, n+m-2)$ and $B'=F(\cc \smallsetminus \{0, 1\}, m-2) $.
The fibre of the fibration $q: E' \to B'$ is the configuration space $F(\cc \smallsetminus \mathcal{O}_m, n)$, which is connected and has homotopical dimension $n$. The homotopical dimension of the base  $F(\cc \smallsetminus \{0, 1\}, m-2)$ is $m-2$. Proposition \ref{prop upper bound} gives $\TC_r[q: E'\to B'] \le
rn+ m-2.$ 
Hence, $$\TC_r[p: F(\cc, n+m) \to F(\cc, m)]\leq rn+ m-2.$$ This completes the proof.
\end{proof}

\begin{remark} Theorems \ref{thm:odd} and \ref{thm:even} leave unanswered the question about the sequential parametrized topological complexity for the Fadell - Neuwirth bundle for $d\ge 4$ even. The upper bound of Proposition \ref{prop upper bound for Fadell-Neuwirth bundle} and the lower bound of Proposition \ref{prop lower bound for Fadell-Neuwirth bundle even d} specify the answer with indeterminacy one. In a forthcoming publication we shall extend the upper bound $rn +m -2$ for any $d\ge 2$ even. We shall employ the method which was briefly described  in \cite{FarW}, \S 7 for the case $r=2$.

\end{remark}

%Now we consider the remaining case $m=2$. We have the equivalence of fibrations $p$ and $q_2$, appearing in the following diagram
%$$\xymatrix{
%F(\cc, n+2) \ar[dr]_{p} \ar[rr]^{h_{2+n}}  &&  F(\cc \smallsetminus \{0, 1\}, n)\times F(\cc, 2) \ar[dl]^{q_2} \\
%& F(\cc, 2) &&
%}$$ where $q_2$ is the projection onto the second factor. Since $q_2$ is the trivial fibration, using Example \ref{example para tc trivial fibration} we have %$$\TC_r[q_2: F(\cc \smallsetminus \{0, 1\}, n)\times F(\cc, 2) \to F(\cc, 2)]=\TC_r(F(\cc \smallsetminus \{0, 1\}, n)).$$This means that,
%$$\TC_r[p: F(\cc, n+2) \to F(\cc, 2)]=\TC_r(F(\cc \smallsetminus \{0, 1\}, n)).$$ The homotopy dimension of $F(\cc \smallsetminus \{0, 1\}, n)$ is $n$, thus $\TC_r(F(\cc \smallsetminus \{0, 1\}, n))\leq rn.$ %Now we use Proposition \ref{prop lower bound for Fadell-Neuwirth bundle even d} to get the lower bound $$rn \leq \TC_r[p: F(\cc, 2+n) \to F(\cc, 2)].$$ 
%Therefore we obtain the desired upper bound for the Fadell-Neuwirth bundle $$\TC_r[p: F(\cc, n+2) \to F(\cc, 2)]\leq rn.$$ This completes the proof.

\section{$\TC$-generating function and rationality}

\subsec{} Definition \ref{def:main} associates with each fibration $p: E\to B$ an infinite sequence of integer numerical invariants, 
\begin{eqnarray}\label{seq}
\TC_2[p: E\to B], \quad \TC_3[p: E\to B], \quad \dots,\quad \TC_r[p: E\to B],\quad \dots
\end{eqnarray}
In order to understand the global behaviour of the sequence (\ref{seq}), it can be organised into a generating function
\begin{eqnarray}\label{gen}
\mathcal F(t) = \sum_{r\ge 1} \TC_{r+1}[p: E\to B]\cdot t^r,
\end{eqnarray}
which we shall call {\it the $\TC$-generating function of the fibration $p: E\to B$}. 
Various analytic properties of the generating function $\mathcal F(t)$ reflect asymptotic behaviour of the sequence (\ref{seq}) and topological structure of the fibration $p: E\to B$. Rationality of the generating function (\ref{gen}) would mean existence of a linear recurrence relation between the integers (\ref{seq}) representing sequential parametrized topological complexities for various values of $r$. 

\begin{lemma}
The $\TC$-generating function (\ref{gen}) depends only on the fiberwise homotopy type of the fibration $p: E\to B$.
\end{lemma}
\begin{proof}
This is equivalent to Corollary \ref{fwhom}. 
\end{proof}

\subsec{} In paper \cite{FarO19} the authors introduced the $\TC$-generating function 
\begin{eqnarray}\label{gen:x}
\mathcal F_X(t) = \sum_{r\ge 1} \TC_{r+1}(X)\cdot t^r
\end{eqnarray}
associating with a finite path-connected CW-complex $X$ a formal power series (\ref{gen:x}). The paper \cite{FarO19} contains many examples when this power series can be explicitly computed and in all these examples 
$\mathcal F_X(t)$ is representable by a rational function of the form 
\begin{eqnarray}\label{form}
\mathcal F_X(t)  = \frac{A}{(1-t)^2} + \frac{B}{1-t}+ p(t), \quad \mbox{where}\quad p(t)\quad \mbox{is a polynomial}.
\end{eqnarray}
This property of $\mathcal F_X(t)$ is equivalent to the recurrence relation
$$\TC_{r+1}(X) = \TC_r(X) +A$$
valid for all sufficiently large $r$; we refer to \cite{FarKS20} for more detail. 
%To explain this fact we mention the well-known power series expansions
%$$(1-t)^{-1} = 1+t+t^2+t^3+\dots\quad\mbox{and}\quad  (1-t)^{-2}=1+2t+3t^2+4t^3+\dots$$
%implying that for every large $r$ the coefficient of $t^{r}$ in the rational function (\ref{form}) equals $rA+B-A$. 
In many examples the principal residue $A$ in (\ref{form}) equals the Lusternik - Schnirelmann category, 
\begin{eqnarray}\label{cat}
A=\cat(X).
\end{eqnarray} 
These examples lead to the Rationality Question of \cite{FarO19}: {\it for which finite CW-complexes the formal power series (\ref{gen:x}) represents a rational function of the form (\ref{form}) with the top residue equal the Lusternik - Schnirelmann category (\ref{cat})?}

\subsec{} In the subsequent paper \cite{FarKS20} the authors analysed a class of CW-complexes violating the Ganea conjecture and found examples $X$ such that the $\TC$-generating function (\ref{gen:x}) is a rational function of the form (\ref{form}) although the top residue $A$ is distinct from $\cat(X)$. 

\subsec{} Next we mention a few examples when the generating function (\ref{gen}) can be computed. 

Firstly, suppose that $p: E\to B$ is the trivial fibration with path-connected fibre $X$. Then the generating function (\ref{gen}) 
equals $\mathcal F_X(t)$. 

Secondly, consider the Hopf bundle $p: S^3\to S^2$. Then, according to Proposition \ref{princ}, we have 
$$\TC_{r+1}[p:S^3\to S^2]=\TC_{r+1}(S^1)= \cat((S^1)^r)=r \quad \mbox{for any}\quad r\ge 1.$$
Therefore, the $\TC$-generating function of the Hopf bundle equals
$$\sum_{r\ge 1} r\cdot t^r \, =\, \frac{t}{(1-t)^2} \, =\, \frac{1}{(1-t)^2} - \frac{1}{1-t}.$$ 
In this case the principal residue equals $A = 1= \cat(S^1)$. 

Exactly the same answer for the $\TC$-generating function can be obtained in the case of a more general Hopf bundle $p: S^{2n+1}\to {\Bbb {CP}}^n$. 

\subsec{} Consider now the $\TC$-generating function of the Fadell - Neuwirth bundle 
$p: F(\rr^d, m+n)\to F(\rr^d, m)$ which was analysed in this paper.  
We start with the case when the dimension $d$ is odd. Then we have the $\TC$-generating function
$$\mathcal F(t) \, =\, \sum_{r=1} [(r+1)n +m-1]\cdot t^r = \frac{n}{(1-t)^2} + \frac{m-1}{1-t} -n - m+1.
$$
It is a rational function of the form (\ref{form}) with the principal residue 
$$A= n = \cat(F(\rr^d-\mathcal O_m, n))$$ equal category of the fibre. Using Theorem 1.3 from \cite{GonG15} we 
may write the $\TC$-generating function of the fiber $X=F(\rr^d-\mathcal O_m, n)$ (for any $d\ge 2$) as follows
\begin{eqnarray}
\mathcal F_X(t) = n \cdot \sum_{r\ge 1} (r+1)t^r = \frac{n}{(1-t)^2} -n. 
\end{eqnarray}

The $\TC$-generating function of the Fadell - Neuwirth bundle is slightly different in the case when $d=2$: 
$$\mathcal F(t) \, =\, \sum_{r=1} [(r+1)n +m-2]\cdot t^r = \frac{n}{(1-t)^2} + \frac{m-2}{1-t} -n - m+2.$$
In this case the power series represents a rational function of form (\ref{form}) and the principal residue equals the Lusternik - Schnirelmann category of the fibre. 

We see that for the Fadell - Neuwirth bundle the $\TC$-generating functions of the bundle and of the fiber have the same principal residue and their difference has a simple pole at $t=1$. 
This suggests the following general question: 

{\it How the $\TC$-generating functions of a fibration $p:E\to B$ and of its fibre $X$ are related?} 
More specifically we may ask: 
{\it For which fibrations $p: E\to B$ the differences 
$$\TC_{r+1}[p:E\to B] - \TC_{r}[p:E\to B] \quad \mbox{and}\quad \TC_{r}[p:E\to B]-\TC_r(X)$$
are eventually constant?}
This stabilisation happens in the case of the Fadell - Neuwirth fibration.

\end{document}